\documentclass[11pt]{amsart}

\usepackage{amsfonts}
\usepackage{amsmath,amssymb,amsthm}
\usepackage{graphicx}
\usepackage[square,numbers]{natbib}
\usepackage{MnSymbol}
\usepackage{mathrsfs,bbm}
\usepackage{url}

\newtheorem{theorem}{Theorem}[section]

\newtheorem{proposition}[theorem]{Proposition}
\newtheorem{corollary}[theorem]{Corollary}
\newtheorem{lemma}[theorem]{Lemma}
\theoremstyle{definition}

\newtheorem{definition}[theorem]{Definition}
\newtheorem{example}[theorem]{Example}
\newtheorem{problem}{Problem}
\newtheorem{remark}[theorem]{Remark}

\begin{document}
\title{Tree-Shifts: The entropy of tree-shifts of finite type}
\keywords{Tree-shift of finite type, entropy, hidden entropy}
\subjclass{Primary 37B10}
\author{Jung-Chao Ban}
\address[Jung-Chao Ban]{Department of Applied Mathematics, National Dong Hwa University, Hualien 970003, Taiwan, ROC.}
\email{jcban@mail.ndhu.edu.tw}
\author{Chih-Hung Chang}
\address[Chih-Hung Chang]{Department of Applied Mathematics, National
University of Kaohsiung, Kaohsiung 81148, Taiwan, ROC.}
\email{chchang@nuk.edu.tw}
\date{March 5, 2017}
\thanks{This work is partially supported by the Ministry of Science and Technology,
ROC (Contract No MOST 105-2115-M-259 -006 -MY2 and 105-2115-M-390 -001 -MY2). The first author is partially supported by National Center for Theoretical Sciences.}

\begin{abstract}
This paper studies the entropy of tree-shifts of finite type with and
without boundary conditions. We demonstrate that computing the entropy of a
tree-shift of finite type is equivalent to solving a system of nonlinear
recurrence equations. Furthermore, the entropy of the binary Markov
tree-shifts over two symbols is either $0$ or $\ln 2$. Meanwhile, the
realization of a class of reals including multinacci numbers is elaborated,
which indicates that tree-shifts are capable of rich phenomena. By
considering the influence of three different types of boundary conditions,
say, the periodic, Dirichlet, and Neumann boundary conditions, the necessary
and sufficient conditions for the coincidence of entropy with and without
boundary conditions are addressed.
\end{abstract}

\maketitle
\baselineskip=1.2\baselineskip



\section{Introduction}


\subsection{History and motivation}

Shift space is a powerful tool to describe the stationary solutions in
physical systems and lattice dynamical systems. Namely, for translation
invariant lattice dynamical systems, stationary solutions could be generated
by posing some local rules in the lattice. One can study the topological
behavior of such systems by studying the shift space instead. The same idea
is also used in dynamical systems. Namely, a map which admits Markov
partition is semi-conjugate to a shift space. The investigation of the
topological properties of such a map is equivalent to exploring the
corresponding shift.


In the classical symbolic dynamical systems, shifts of finite type are an
important class for the purpose of illustration. A shift of finite type is
realized as a set of infinite paths in a finite directed graph. In the past
three decades, the study of the complexity for the given systems has
received considerable attention. Entropy plays an important role to measure
such a property. In information theory, entropy also measures the
\textquotedblleft information capacity\textquotedblright\ or
\textquotedblleft ability to transmit messages\textquotedblright . 
Readers are referred to \cite{ALM-2000, Downarowicz-2011, Kit-1998, LM-1995,
P-1997} and the references therein for more details.

While the dynamical behavior of shifts of finite type is explicitly
characterized, the properties of multidimensional symbolic dynamical systems
are barely known. For the classical symbolic dynamics, Williams indicates
that the conjugacy of one-sided shifts of finite type is decidable, and
topological entropy can be treated as an indicator for classifying shifts of
finite types (see \cite{LM-1995}). Nevertheless, the conjugacy of
multidimensional shifts of finite type is undecidable (see \cite%
{CJJ+-ETDS2003, JM-PAMS1999, LS-2002}). Furthermore, there is no algorithm
for the computation of topological entropy of multidimensional shifts of
finite type (see \cite{BL-DCDS2005, BPS-TAMS2010, HM-AoM2010, PS-JdM2015}
and the references therein).

In \cite{AB-TCS2012, AB-TCS2013}, the authors introduce the notion of shifts
defined on infinite trees, that are called tree-shifts. Infinite trees have
a natural structure of one-sided symbolic dynamical systems equipped with
multiple shift maps. The $i$th shift map applies to a tree which gives the
subtree rooted at the $i$th children of the tree. Sets of finite patterns of
tree-shifts of finite type are strictly testable tree languages. Such
testable tree languages are also called $k$-testable tree languages.
Probabilistic $k$-testable models are used for pattern classification and
stochastic learning \cite{VCC-PAMIIEEET2005}.

Tree-shifts are interesting for elucidation due to them constituting an
intermediate class in between one-sided shifts and multidimensional shifts.
Aubrun and B\'{e}al extend Williams' result to tree-shifts; more precisely,
they show that the conjugacy of irreducible tree-shifts of finite type is
decidable \cite{AB-TCS2012}. Furthermore, Aubrun and B\'{e}al accomplish
other celebrated results in tree-shifts, such as realizing tree-shifts of
finite type and sofic tree-shifts via tree automata, developing an algorithm
for determining whether a sofic tree-shift is a tree-shift of finite type,
and the existence of irreducible sofic tree-shifts that are not factors of
tree-shifts of finite type. Readers are referred to \cite{AB-TCS2012,
AB-TCS2013} for more details.

This paper, which is a consequential investigation of \cite{BC-2015},
focuses on the entropy of tree-shifts of finite type and intends to reveal
the connection between tree-shifts of finite type and nonlinear dynamical
systems. We demonstrate that the characterization of entropy of tree-shifts
of finite type relates to solving systems of nonlinear recurrence equations
(Theorems \ref{Thm : 1} and \ref{thm:1to1-correspondence-entropy-SNRE}).
Meanwhile, the entropy of binary Markov tree-shifts over the alphabet of two
symbols is elaborated completely (Theorem \ref{Thm: 3}) while a large set of
reals can be realized via Markov tree-shifts over an alphabet of more than
two symbols (Theorem \ref{Thm : 2}).

Notably, the investigation of entropy of shifts of finite type is mainly
relevant to the rules of allowed patterns and can be treated as a subclass
of recurrence equations. The computation is rather difficult for
multidimensional shifts. In other words, the effect of dimension of the
underlying spaces is only seen in the higher dimensional cases. The
difficulty also surrounds the elucidation of entropy of tree-shifts of
finite type. More explicitly, computing entropy of tree-shifts is equivalent
to solving systems of nonlinear recurrence equations, but it is known that
the computation of entropy of shifts of finite type is investigating linear
recurrence equations instead. Rather than the novelty of elaboration of
systems of nonlinear recurrence equations, there are many discussion about
nonlinear recurrence equations, readers are referred to \cite{GreathouseIV-2012} for more details. Remarkably,
a $d$-dimensional tiling system is a tree-shifts of finite type with abelian
nodes and each node has $d$ children. The related discussion is in
preparation.

Meanwhile, boundary conditions are commonly considered in lattice dynamical
systems and three types are typically considered. Namely, \emph{Dirichlet (}%
or\emph{\ fixed)}, \emph{Neumann (}or\emph{\ zero-flux)}, and \emph{periodic}
boundary conditions. (See \cite{BB-N2006, CMV-RCD1996, CR-2002, HR-1985,
Kee-SJAM1987, MC-ITCSIFTaA1995, Shi-SJAM2000, Thir-ITCSIFTA1993} for
examples.) Traditionally, a Dirichlet condition means that such a system is
forced to have a fixed input and output by energy in the boundary; a Neumann
condition means that there is no input in the boundary cell due to the fact
that any input would cause energy and/or material flow from the outside
making the system an \textquotedblleft open system\textquotedblright\ in the
sense of thermodynamics in physics; the periodic boundary condition is
equivalent to fabricating a chip on a silicon torus as its substrate.

For a system with a boundary constraint, it is interesting to see the effect
from various boundary conditions. Correspondingly, the following question
arose: 
\begin{equation}
h=h^{N}=h^{P}=h^{D}\text{?}  \label{13}
\end{equation}%
Herein, $h$ relates to the entropy, and $h^{N},h^{P}$, and $h^{D}$ relates
to the entropy constrained with Neumann, periodic, and Dirichlet boundary
conditions, respectively. Such a problem was posed in \cite{AH-2003} and the
references therein. Intrinsically, it indicates that the complexity of a
given system will not be influenced by the boundary if the equality holds.
Otherwise it will be influenced and $h-h^{i}$ measures the difference. In
circuit theory, it also means that we could control the \textquotedblleft
ideal system\textquotedblright\ (i.e. without a boundary condition) by
manipulating the input in the boundary (see \cite{CR-2002} for more
details). Theorem \ref{thm:Entropy-NBC} (resp. Theorem \ref{thm:Entropy-DBC}
and Theorem \ref{thm:Entropy-PBC}) provides a natural and intrinsic
characterization of question (\ref{13}) for $h=h^{N}$ (resp. $h=h^{D}$ and $%
h=h^{P}$).

The upcoming subsection elucidates the correspondence between Markov
tree-shifts and systems of nonlinear recurrence equations, followed by the
analysis of the entropy of binary Markov tree-shifts with two symbols.
Section 3 addresses the realization theorem of tree-shifts of finite type.
The boundary influence for the entropy is elucidated in Section 4.
Discussion and conclusions are given in Section 5.


\subsection{Preliminaries and notations}

This subsection recalls some basic definitions of symbolic dynamics on
infinite trees. The nodes of infinite trees considered in this paper have a
fixed number of children and are labeled in a finite alphabet. Hence the
class of classical one-sided shift spaces is a special case in the present
paper.

Let $\Sigma = \{0, 1, \ldots, d-1\}$ and let $\Sigma^* = \bigcup_{n \geq 0}
\Sigma^n$ be the set of words over $\Sigma$, where $\Sigma^n = \{w_1 w_2
\cdots w_n: w_i \in \Sigma \text{ for } 1 \leq i \leq n\}$ is the set of
words of length $n$ for $n \in \mathbb{N}$ and $\Sigma^0 = \{\epsilon\}$
consists of the empty word $\epsilon$. An \emph{infinite tree} $t$ over a
finite alphabet $\mathcal{A}$ is a function from $\Sigma^*$ to $\mathcal{A}$%
. A node of an infinite tree is a word of $\Sigma^*$. It is seen that the
empty word relates to the root of the tree. Suppose $x$ is a node of a tree. 
$x$ has children $xi$ with $i \in \Sigma$. A sequence of words $(x_k)_{1
\leq k \leq n}$ is called a \emph{path} if, for all $k \leq n-1$, $x_{k+1} =
x_k i_k$ for some $i_k \in \Sigma$.

Let $t$ be a tree and let $x$ be a node, we refer $t_x$ to $t(x)$ for
simplicity. A subset of words $L \subset \Sigma^*$ is called \emph{%
prefix-closed} if each prefix of $L$ belongs to $L$. A function $u$ defined
on a finite prefix-closed subset $L$ with codomain $\mathcal{A}$ is called a 
\emph{pattern} (or \emph{block}), and $L$ is called the \emph{support} of
the pattern. A subtree of a tree $t$ rooted at a node $x$ is the tree $%
t^{\prime }$ satisfying $t^{\prime }_y = t_{xy}$ for all $y \in \Sigma^*$
such that $xy$ is a node of $t$, where $xy = x_1 \cdots x_m y_1 \cdots y_n$
means the concatenation of $x = x_1 \cdots x_m$ and $y_1 \cdots y_n$.

Suppose $n$ is a nonnegative integer. Let $\Sigma_n = \bigcup_{k = 0}^n
\Sigma^k$ denote the set of words of length at most $n$. We say that a
pattern $u$ is \emph{a block of height $n$} (or \emph{an $n$-block}) if the
support of $u$ is $\Sigma_{n-1}$, denoted by $\mathrm{height}(u) = n$.
Furthermore, $u$ is a pattern of a tree $t$ if there exists $x \in \Sigma^*$
such that $u_y = t_{xy}$ for every node $y$ of $u$. In this case, we say
that $u$ is a pattern of $t$ rooted at the node $x$. A tree $t$ is said to 
\emph{avoid} $u$ if $u$ is not a pattern of $t$. If $u$ is a pattern of $t$,
then $u$ is called an \emph{allowed pattern} of $t$.

We denote by $\mathcal{T}$ (or $\mathcal{A}^{\Sigma^*}$) the set of all
infinite trees on $\mathcal{A}$. For $i \in \Sigma$, the shift
transformations $\sigma_i$ from $\mathcal{T}$ to itself are defined as
follows. For every tree $t \in \mathcal{T}$, $\sigma_i(t)$ is the tree
rooted at the $i$th child of $t$, that is, $\sigma_i(t)_x = t_{ix}$ for all $%
x \in \Sigma^*$. For the simplification of the notation, we omit the
parentheses and denote $\sigma_i(t)$ by $\sigma_i t$. The set $\mathcal{T}$
equipped with the shift transformations $\sigma_i$ is called the \emph{full
tree-shift} of infinite trees over $\mathcal{A}$. Suppose $w = w_1 \cdots
w_n \in \Sigma^*$. Define $\sigma_w = \sigma_{w_n} \circ \sigma_{w_{n-1}}
\circ \cdots \circ \sigma_{w_1}$. It follows immediately that $(\sigma_w
t)_x = t_{wx}$ for all $x \in \Sigma^*$.

Given a collection of patterns $\mathcal{F}$, let $\mathsf{X}_{\mathcal{F}}$
denote the set of all trees avoiding any element of $\mathcal{F}$. A subset $%
X \subseteq \mathcal{T}$ is called a \emph{tree-shift} if $X = \mathsf{X}_{%
\mathcal{F}}$ for some $\mathcal{F}$. We say that $\mathcal{F}$ is \emph{a
set of forbidden patterns} (or \emph{a forbidden set}) of $X$. It can be
seen that a tree-shift satisfies $\sigma_w X \subseteq X$ for all $w \in
\Sigma^*$. A tree-shift $X = \mathsf{X}_{\mathcal{F}}$ is called a \emph{%
tree-shift of finite type} (TSFT) if the forbidden set $\mathcal{F}$ is
finite.

Denote the set of all blocks of height $n$ of $X$ by $B_n(X)$, and denote
the set of all blocks of $X$ by $B(X)$. Suppose $u \in B_n(X)$ for some $n
\geq 2$. Let $\sigma_i u$ be the block of height $n-1$ such that $(\sigma_i
u)_x = u_{ix}$ for $x \in \Sigma_{n-2}$. The block $u$ is written as $u =
(u_{\epsilon}, \sigma_0 u, \sigma_1 u, \ldots, \sigma_{d-1} u)$.

Let $\mathcal{B}$ be a collection of $2$-blocks, and let $X^{\mathcal{B}} = 
\mathsf{X}_{\mathcal{F}}$, where $\mathcal{F} = B_2(\mathcal{T}) \setminus 
\mathcal{B}$. It follows immediately that $X^{\mathcal{B}}$ is a TSFT. We
say that $\mathcal{B}$ is \emph{a basic set of allowed patterns} (or \emph{a
basic set}) of $X^{\mathcal{B}}$ and $X^{\mathcal{B}}$ is generated by $%
\mathcal{B}$. It can be seen that a tree-shift satisfies $\sigma_w X^{%
\mathcal{B}} \subseteq X^{\mathcal{B}}$ for all $w \in \Sigma^*$. In
particular, for the case where $d = 2$, a $2$-block $u \in \mathcal{B}$ is
written as $(i, j, k)$.

We remark that a basic set of allowed patterns is, in general, a subset of $%
B_n(\mathcal{T})$ for some $n \in \mathbb{N}$. For the clarification of the
investigation, we consider $\mathcal{B} \subseteq B_2(\mathcal{T})$ for the
rest of this paper unless otherwise stated. A TSFT is called a \emph{Markov
tree-shift} in this case. One of the most frequently used quantum which
describes the complexity of a dynamical system is \emph{entropy}. For a
tree-shift $X^{\mathcal{B}}$, the definition of entropy is given as follows.

\begin{definition}
Suppose a basic set $\mathcal{B}$ is given.

\begin{enumerate}
\item The \emph{entropy} of $X^{\mathcal{B}}$, denoted by $h(X^{\mathcal{B}%
})=h(\mathcal{B})$, is defined as 
\begin{equation}
h(\mathcal{B})=\lim_{n\rightarrow \infty }\frac{\ln ^{2}|B_{n}(X^{\mathcal{B}%
})|}{n},  \label{1}
\end{equation}%
whenever the limit exists, where $|B_{n}(X^{\mathcal{B}})|$ means the
cardinality of $B_{n}(X^{\mathcal{B}})$, and $\ln ^{2}=\ln \circ \ln $.

\item If $|B_{n}(X^{\mathcal{B}})|$ behaves like $\exp (\alpha \kappa ^{n})$%
, such as $c_{1}\exp (\alpha \kappa ^{n})\leq |B_{n}(X^{\mathcal{B}})|\leq
c_{2}\exp (\alpha \kappa ^{n})$ for some $c_{1},c_{2}>0$ and $\alpha ,\kappa
\geq 0$ are constants. From (\ref{1}) we have $h(\mathcal{B})=\ln \kappa $,
and called $\alpha $ the \emph{hidden entropy} (or \emph{sub-entropy}) of $%
X^{\mathcal{B}}$.
\end{enumerate}
\end{definition}

It is seen that the topological entropy of a full tree-shift $\mathcal{T} = \mathcal{A}^{\Sigma^*}$ is $h(\mathcal{T}) = \ln d$, where $d = |\Sigma| \in \mathbb{N}$ with $d \geq 2$. Indeed, the cardinality of $n$-blocks in $\mathcal{T}$ is $|B_n(\mathcal{T})| = k^{\frac{d^n-1}{d-1}}$, where $k = |\mathcal{A}|$. It follows that
$$
h(\mathcal{T}) = \lim_{n \to \infty} \dfrac{\ln^2 |B_n(\mathcal{T})|}{n} = \lim_{n \to \infty} \dfrac{\ln^2 k + \ln (d^n - 1) - \ln (d-1)}{n} = \ln d.
$$

The following lemma addresses the sufficient condition for the existence of the limit \eqref{1}. The proof is straightforward, thus it is omitted.

\begin{lemma}\label{thm:existence-entropy-limit}
The limit \eqref{1} exists if $\displaystyle\lim_{n\rightarrow \infty }\frac{\ln |B_{n}(X^{\mathcal{B}})|}{d^{n}}>0$.
\end{lemma}

%
%
%

Aside from the existence of the limit \eqref{1}, it is natural to ask the
following question.

\begin{problem}
\label{Prob: 1}How can the hidden entropy be computed?
\end{problem}

This work focuses on the investigation of \eqref{1}. The study of Problem %
\ref{Prob: 1} is in preparation. It is seen below that, for those
\textquotedblleft good enough\textquotedblright\ basic sets $\mathcal{B}$,
the entropy $h(\mathcal{B})$ can be computed explicitly.

Suppose $X^{\mathcal{B}}$ is a one-sided shift space, i.e., consider the
case where $d = 1$. The entropy of $X^{\mathcal{B}}$ is defined as 
\begin{equation*}
h(\mathcal{B}) = \lim_{n\rightarrow \infty }\frac{\ln |B_n(X^{\mathcal{B}})|%
}{n}.
\end{equation*}
Notably, in this case, $B_n(X^{\mathcal{B}})$ is a subset of $\mathcal{A}^n$%
. This makes the growth rate of $|B_n(X^{\mathcal{B}})|$ with respect to $n$
up to exponential, and the existence of the limit comes from the
subadditivity of $\{\ln |B_n(X^{\mathcal{B}})|\}_{n \geq 1}$ (see \cite%
{LM-1995}). For the case where $X^{\mathcal{B}}$ is a strict tree-shift,
i.e., $d \geq 2$, the growth rate of $|B_n(X^{\mathcal{B}})|$ with respect
to $n$ becomes doubly exponential since the size of the support of an $n$%
-block increases doubly exponentially.

Another difference between the topological entropy of one-sided shift spaces
and tree-shifts is the introduction of hidden entropy. It is remarkable
that, for a strict tree-shift rather than a one-sided shift space, the
hidden entropy means the ``average number of symbols'' used while the
topological entropy infers the ``average spatial dimension'' of the
underlying lattice. An intuitive observation can be seen from the full
tree-shift. Suppose $\mathcal{T}$ is the full tree-shift over $|\mathcal{A}|
= k$ and each node has exactly $d$ children. Then the entropy and hidden
entropy of $\mathcal{T}$ is $\ln d$ and $\ln k/(d-1)$, respectively.

Let $\mathcal{A}=\{x^{(1)}, x^{(2)}, \ldots ,x^{(k)}\}$ be the (ordered) symbol set and $X_{x^{(i)}}^{\mathcal{B}}$ denote the collection of trees $t$ satisfying$t_{\epsilon }=x^{(i)}$, where $1\leq i\leq k$. Given $n \in \mathbb{N}$ with $n \geq 2$, let $\mathcal{A}^n$ be an ordered set with respect to the lexicographic order. More explicitly, for any two $n$-words $\mathbf{x} = x_1 x_2 \cdots x_n$ and $\mathbf{y} = y_1 y_2 \cdots y_n$, we say that $\mathbf{x}<\mathbf{y}$, if and only if there
exists $1 < r \leq n$ such that $x_{i}=y_{i}$ for $i\leq r-1$ and $%
x_{r}<y_{r}$, and $\mathbf{x}=\mathbf{y}$ otherwise.

\begin{definition}
\label{Def: order and indicator vector} Let $\mathcal{A}=\{x^{(1)},\ldots
,x^{(k)}\}$ be the symbol set and suppose $\mathcal{A}^{d}$ is an ordered
set with respect to the lexicographic order, $d \in \mathbb{N}$.
\begin{enumerate}
\item Let $F(x^{(1)},\ldots ,x^{(k)})=\sum\limits_{\mathbf{x}\in \mathcal{A}^{d}} f_{\mathbf{x}}\mathbf{x}$ be a binary combination over $\mathcal{A}^{d}$, i.e., $f_{\mathbf{x}} \in \{0, 1\}$ for $\mathbf{x} \in \mathcal{A}^{d}$. The vector $v_{F}=(f_{\mathbf{x}})_{\mathbf{x}\in \mathcal{A}^{d}} \in \mathbb{R}^{k^d}$ is called the \emph{indicator vector} of $F$.
\item Suppose $v_F$ and $v_G$ are the indicator vectors of $F=\sum\limits_{\mathbf{x}\in \mathcal{A}^{d}} f_{\mathbf{x}}\mathbf{x}$ and $G=\sum\limits_{\mathbf{x}\in \mathcal{A}^{d}} g_{\mathbf{x}}\mathbf{x}$, respectively. We say that $v_F$ \emph{dominates} $v_G$, denoted by $v_F \geq v_G$, if $f_{\mathbf{x}} \geq g_{\mathbf{x}}$ for all $\mathbf{x} \in \mathcal{A}^{d}$. In other words, $v_F - v_G$ is a nonnegative vector.
\end{enumerate}
\end{definition}

We remark that, since $\mathcal{A}^{d}$ is an ordered set with respect to the lexicographic order, each binary combination $F(x^{(1)},\ldots ,x^{(k)})=\sum\limits_{\mathbf{x}\in \mathcal{A}^{d}} f_{\mathbf{x}}\mathbf{x}$ is also an ordered combination. For instance, $v_{(x^{(1)})^{d}}=(1,0, \ldots ,0)$, and $v_{F}=e_{k^{d}}$ if $F=\sum\limits_{\mathbf{x}\in \mathcal{A}^{d}} \mathbf{x}$, where $e_{k^{d}}$ is the vector whose entries are all $1$'s.

\begin{definition}
A sequence $\{a_{n}^{(1)},\ldots ,a_{n}^{(k)}\}_{n\in \mathbb{N}}$ is defined by a \emph{%
system of nonlinear recurrence equations} (SNRE) of degree $(d,k)$ if 
\begin{equation*}
a_{n}^{(i)}=F^{(i)}(a_{n-1}^{(1)},a_{n-1}^{(2)},\ldots ,a_{n-1}^{(k)})\quad 
\text{for}\quad n\geq 2,1\leq i\leq k,
\end{equation*}%
and $a_{1}^{(i)}=a^{(i)}$ for $1\leq i\leq k$, where $F^{(1)}, \ldots, F^{(k)}$ are binary combinations over $\{a_{n-1}^{(1)},a_{n-1}^{(2)},\ldots ,a_{n-1}^{(k)}\}^d$, respectively.
\end{definition}

For the case where $k=2$, set $\mathcal{A}=\{x^{(1)},x^{(2)}\}=\{1,2\}$, $%
a_{n}(\mathcal{B})=|B_{n}(X_{1}^{\mathcal{B}})|$, $b_{n}(\mathcal{B}%
)=|B_{n}(X_{2}^{\mathcal{B}})|$ and $c_{n}(\mathcal{B})=|B_{n}(X^{\mathcal{B}%
})|$. We shall usually shorten $a_{n}(\mathcal{B}),b_{n}(\mathcal{B})$, and $%
c_{n}(\mathcal{B})$ to $a_{n},b_{n}$, and $c_{n}$, respectively, when $%
\mathcal{B}$ is understood. For each $n\in \mathbb{N}$, let $\Gamma
_{n}:\{1,2\}\rightarrow \{a_{n},b_{n}\}$ be defined as $\Gamma _{n}(1)=a_{n}$
and $\Gamma _{n}(2)=b_{n}$. We have the following theorem.

\begin{theorem}[SNRE for $d=k=2$]
\label{Thm : 1} Given a basic set $\mathcal{B}$; then $c_{n}=a_{n}+b_{n}$
and $a_{n}$ and $b_{n}$ satisfy the following SNRE of degree $(2,2)$: 
\begin{equation*}
\left\{ 
\begin{array}{l}
a_{n}=\sum\limits_{(1,j,k)\in \mathcal{B}}\Gamma _{n-1}(j)\Gamma _{n-1}(k),
\\ 
b_{n}=\sum\limits_{(2,j,k)\in \mathcal{B}}\Gamma _{n-1}(j)\Gamma _{n-1}(k),
\\ 
a_{2}=|B_{2}(X_{1}^{\mathcal{B}})|,b_{2}=|B_{2}(X_{2}^{\mathcal{B}})|.%
\end{array}%
\right.
\end{equation*}
\end{theorem}

\begin{proof}
It can be easily checked that $c_{n}=a_{n}+b_{n}$ and, if $u=(i,j,k)\in 
\mathcal{B}$, then the numbers of $n$-blocks that has $u$ rooted at $%
\epsilon $ is $\Gamma _{n-1}(j)\Gamma _{n-1}(k)$. Therefore, we have 
\begin{equation*}
a_{n}=\sum\limits_{(1,j,k)\in \mathcal{B}}\Gamma _{n-1}(j)\Gamma
_{n-1}(k),b_{n}=\sum\limits_{(2,j,k)\in \mathcal{B}}\Gamma _{n-1}(j)\Gamma
_{n-1}(k),\quad n\geq 2,
\end{equation*}%
with the initial condition $a_{2}=|B_{2}(X_{1}^{\mathcal{B}})|$ and $%
b_{2}=|B_{2}(X_{2}^{\mathcal{B}})|$. This completes the proof.
\end{proof}

Theorem \ref{Thm : 1} indicates that each basic set of $2$-blocks associates
with an SNRE of degree $(2, 2)$. In general, Theorem \ref%
{thm:1to1-correspondence-entropy-SNRE}, which is an extension of Theorem \ref%
{Thm : 1}, reveals that the calculation of the topological entropy of a
tree-shift of finite type is equivalent to solving a system of nonlinear
recurrence equations.

\begin{theorem}
\label{thm:1to1-correspondence-entropy-SNRE} The topological entropy of a
tree-shift of finite type is realized as a system of nonlinear recurrence
equations of degree $(d, k)$ for some $d, k$. Conversely, every system of
nonlinear recurrence equations of degree $(d, k)$ corresponds to the
topological entropy of some tree-shifts of finite type.
\end{theorem}

\begin{proof}
Suppose $X$ is a tree-shift of finite type. Without loss of generality, we
may assume that $X = \mathsf{X}_{\mathcal{F}}$ is a Markov tree-shift over $%
\mathcal{A} = \{1, 2, \ldots, k\}$ with exactly $d$ children for each node
of $t \in X$. Recall that $X$ is called a Markov tree-shift if $\mathcal{F}$
consists of $2$-blocks.

For $1 \leq i \leq k$, let $a^{(i)}_n = |B_n(X_i)|$
be the cardinal number of $n$-blocks of $X_i$ for $n \in \mathbb{N}$. Then $%
|B_n(X)| = \Sigma_{i=1}^k a^{(i)}_n$. Write $u \in B_2(X)$ as $u =
(u_{\epsilon}, u_0, \ldots, u_{d-1})$. It is easily seen that 
\begin{equation*}
a^{(i)}_n = \sum_{u \in B_2(X_i)} a^{(u_0)}_{n-1} a^{(u_1)}_{n-1} \cdots
a^{(u_{d-1})}_{n-1}, \quad \text{for} \quad 1 \leq i \leq k, n \geq 3.
\end{equation*}
Define 
\begin{equation*}
F^{(i)}(\alpha_1, \alpha_2, \ldots, \alpha_k) = \sum_{u \in B_2(X_i)}
\alpha^{(u_0)}_{n-1} \alpha^{(u_1)}_{n-1} \cdots \alpha^{(u_{d-1})}_{n-1}.
\end{equation*}
It follows that the sequence $\{a^{(1)}_n, a^{(2)}_n, \ldots, a^{(k)}_n\}_{n
\geq 2}$ satisfies the SNRE 
\begin{equation*}
\left\{ \begin{aligned} a^{(i)}_n &= F^{(i)}(a^{(1)}_{n-1}, a^{(2)}_{n-1},
\ldots, a^{(k)}_{n-1}), \quad n \geq 3; \\ a^{(i)}_2 &= |B_2(X_i)|, \quad 1
\leq i \leq k. \end{aligned} \right.
\end{equation*}
Hence, the topological entropy of $X$ is realized by the above SNRE.

Conversely, suppose $\{a^{(1)}_n, a^{(2)}_n, \ldots, a^{(k)}_n\}_{n \in 
\mathbb{N}}$ satisfies the following SNRE of degree $(d, k)$: 
\begin{equation*}
\left\{ \begin{aligned} a^{(i)}_n &= f_i(a^{(1)}_{n-1}, a^{(2)}_{n-1},
\ldots, a^{(k)}_{n-1}), \quad n \geq 2; \\ a^{(i)}_1 &= a_i, \quad 1 \leq i
\leq k. \end{aligned} \right.
\end{equation*}
Set 
\begin{equation*}
\mathcal{B}_i = \{(i, p_1, \ldots, p_d): a_{\mathbf{x}} \neq 0, \text{ where 
} f_i = \Sigma a_{\mathbf{x}} \mathbf{x}, \mathbf{x} = \alpha_{p_1} \ldots
\alpha_{p_d} \in \Lambda^d\},
\end{equation*}
where $\Lambda = \{\alpha_1, \ldots, \alpha_k\}$. Let $\mathcal{B} =
\bigcup_{i=1}^k \mathcal{B}_i$ and let $\mathcal{F} = \mathcal{A}^2
\setminus \mathcal{B}$, herein $\mathcal{A} = \{1, \ldots, k\}$.

Similar to the discussion above, it can be verified without difficulty that $%
X = \mathsf{X}_{\mathcal{F}}$ is a Markov tree-shift and $B_n(X_i) =
a^{(i)}_{n-1}$ for $n \geq 3$ and $1 \leq i \leq k$.

This completes the proof.
\end{proof}


\section{Entropy for $d=k=2$}

Let $d=k=2$ and let $\mathcal{B}$ be a basic set which associated with an
SNRE (see Theorem \ref{Thm : 1}). In this case, we have a pair of indicator
vectors $(v_{F},v_{G})$ (see Definition \ref{Def: order and indicator vector}%
), where 
\begin{align*}
F(a_{n-1},b_{n-1})& =\sum\limits_{(1,j,k)\in \mathcal{B}}\Gamma
_{n-1}(j)\Gamma _{n-1}(k), \\
G(a_{n-1},b_{n-1})& =\sum\limits_{(2,j,k)\in \mathcal{B}}\Gamma
_{n-1}(j)\Gamma _{n-1}(k),
\end{align*}%
and $\Gamma _{n}(1)=a_{n},\Gamma _{n}(2)=b_{n}$, respectively. An SNRE is of
the \emph{dominant-type} if either $v_{F}$ dominates $v_{G}$ or $v_{G}$
dominates $v_{F}$. Finally we define $n_{\ast }:=|\{i:(v_{\ast })_{i}\neq
0\}|$, where $\ast $ stands for $F$ or $G$ and $\left( v\right) _{i}$ is the 
$i$-coordinate of $v$. Such number counts the number of nonzero entries in $%
v_{F}$ and $v_{G}$, respectively.


\subsection{Classical results}

In \cite{AS-FQ1973}, Aho and Sloane show that a sequence of natural numbers $%
\{x_i\}_{i \geq 1}$ satisfying a nonlinear recurrence of the form $x_{n+1} =
x_n^2 + g_n$, with $|g_n|< \frac{1}{4} x_n$ for $n \geq n_0$ has doubly
exponential form of $x_n \approx k^{2^n}$ for some $k$. The constant $k$ is
unknown in general. Ionascu and Stanica \cite{IS-AMUC2004} extend Aho and
Sloane's result and formulate $x_n$ explicitly for some $g_n$. Many
researchers have devoted to the study of nonlinear recurrence equations and
only a few results are obtained (cf.~%
\cite{GreathouseIV-2012}). Theorem \ref%
{thm:1to1-correspondence-entropy-SNRE} infers the corresponding between SNRE
and tree-shifts of finite type, which reveals the difficulty of the
computation of entropy.

\begin{lemma}
\label{Prop : 1} Suppose, for $n \geq 1$, $x_n$ satisfies a nonlinear
recurrence equation 
\begin{equation*}
\left\{ 
\begin{array}{l}
x_{n+1}=x_{n}^{2}+\left\vert g_{n}\right\vert , \\ 
\left\vert g_{n}\right\vert \leq x_{n}, \\ 
x_1 = x^1.%
\end{array}%
\right.
\end{equation*}%
Then 
\begin{equation*}
\lim_{n\rightarrow \infty }\frac{\ln ^{2}x_{n}}{n}=\ln 2.
\end{equation*}
\end{lemma}

%


\subsection{Complete characterization of entropy for $d=k=2$}

This subsection elaborates the entropy of tree-shifts of finite type for the
case where $d = k = 2$. Further investigation is discussed in the upcoming
section.

\begin{lemma}
\label{Lma : 2} If an SNRE corresponding to a given $\mathcal{B}$ is of the
dominant-type and satisfies one of the following conditions:
\begin{enumerate}
\item $(v_{F})_{1}\neq 0$ and $n_{F}\geq 2$,

\item $(v_{F})_{1}=0$, and $n_{F}$ and $n_{G}\geq 2$,
\end{enumerate}
then $h(\mathcal{B})=\ln 2$.
\end{lemma}

\begin{proof}
Without loss of generality, we assume $v_{F}\geq v_{G}$. It follows that $%
a_{n}\geq b_{n}$ for all $n\geq 2$ and 
\begin{equation*}
h(\mathcal{B})=\lim_{n\rightarrow \infty }\frac{\ln ^{2}c_{n}}{n}%
=\lim_{n\rightarrow \infty }\frac{\ln ^{2}\left( a_{n}+b_{n}\right) }{n}%
=\lim_{n\rightarrow \infty }\frac{\ln ^{2}a_{n}}{n}\text{.}
\end{equation*}%
Thus, it suffices to compute $\lim_{n\rightarrow \infty }\dfrac{\ln ^{2}a_{n}%
}{n}$. We divide the proof into two cases.

\noindent \textbf{1.} $(v_{F})_{1}\neq 0$. A similar discussion to the proof
of Lemma \ref{Prop : 1} indicates that $h(\mathcal{B})=\ln 2.$

\noindent \textbf{2.} $(v_{F})_{1}=0$. We divide this case into several
sub-cases.

\noindent \textbf{2-a.} $v_{F}=(0,1,1,1)$. In this case, since $n_{G}\geq 2$%
, it is seen that $v_{G}=(0,1,1,0)$, $(0,0,1,1)$, or $(0,1,0,1)$. For the
case where $v_{G}=(0,1,1,0)$, we define a SNRE as follows. 
\begin{equation*}
\left\{ 
\begin{array}{l}
d_{n}=d_{n-1}e_{n-1}+e_{n-1}d_{n-1}, \\ 
e_{n}=d_{n-1}e_{n-1}+e_{n-1}d_{n-1}, \\ 
d_{2}=e_{2}=2.%
\end{array}%
\right.
\end{equation*}%
Then we have $d_{n}=e_{n}$, $a_{n}\geq d_{n}$, and $b_{n}\geq e_{n}$ for $%
n\geq 2$, which yields that 
\begin{equation*}
\ln 2\geq h(\mathcal{B})=\lim_{n\rightarrow \infty }\frac{\ln ^{2}c_{n}}{n}%
=\lim_{n\rightarrow \infty }\frac{\ln ^{2}a_{n}}{n}\geq \lim_{n\rightarrow
\infty }\frac{\ln ^{2}d_{n}}{n}=\ln 2.
\end{equation*}%
Hence $h(\mathcal{B})=\ln 2$. Suppose $v_{G}=(0,0,1,1)$. Define a SNRE as
follows. 
\begin{equation*}
\left\{ 
\begin{array}{l}
d_{n}=d_{n-1}e_{n-1}, \\ 
e_{n}=e_{n-1}d_{n-1}+e_{n-1}^{2}, \\ 
d_{2}=1\text{ and }e_{2}=2.%
\end{array}%
\right.
\end{equation*}%
Thus we have $a_{n}\geq d_{n}$, $b_{n}\geq e_{n}\geq d_{n}$ for $n\geq 2$.
Since $v_{F}>v_{G}$, it follows that 
\begin{align*}
\ln 2\geq h(\mathcal{B})& =\lim_{n\rightarrow \infty }\frac{\ln ^{2}c_{n}}{n}%
=\lim_{n\rightarrow \infty }\frac{\ln ^{2}a_{n}}{n}\geq \lim_{n\rightarrow
\infty }\frac{\ln ^{2}d_{n}}{n} \\
& \geq \lim_{n\rightarrow \infty }\frac{\ln ^{2}e_{n-1}}{n}=\ln 2.
\end{align*}%
Note that the same proof of case \textbf{1} works for $\lim_{n\rightarrow
\infty }\frac{\ln ^{2}e_{n-1}}{n}=\ln 2$. The case where $v_{G}=(0,1,0,1)$
can be demonstrated analogously. This completes the proof.

\noindent \textbf{2-b.} $v_{F}=(0,0,1,1)$ or $(0,1,0,1).$ We deal with the
case $v_{F}=(0,0,1,1)$. The other can be indicated similarly. Since $%
n_{G}\geq 2$, $v_{G}$ has only one possibility $v_{G}=(0,0,1,1)$. It follows
immediately that $a_{n}=b_{n}$ for all $n$ and the SNRE is transformed into $%
a_{n}=2a_{n-1}^{2},$ $a_{2}\geq 2.$ Thus $h(\mathcal{B})=\ln 2.$

\noindent \textbf{2-c.} $v_{F}=(0,1,1,0).$ In this case, $v_{G}=(0,1,1,0)$.
Applying the same argument given in \textbf{2-b} yields that $h(\mathcal{B}%
)=\ln 2$. The proof is thus complete.
\end{proof}

\begin{remark}
\label{Rmk: 1} If we soften the condition (ii) in Lemma \ref{Lma : 2} to $%
n_{G}<2$, then, in the proof of \textbf{2-a}, we only need to deal with $%
v_{G}=(0,0,0,1)$ (the cases where $v_{G}=(0,1,0,0)$ and $v_{G}=(0,0,1,0)$
are similar). In this case, it is seen that $b_{n}=1$ for all $n\geq 2$ and
the equation becomes $a_{n}=2a_{n-1}+a_{n-1}^{2}$. In this case we still
have $h(\mathcal{B})=\ln 2$. In \textbf{2-b}, it remains to elaborate the
case where $v_{F}=(0,0,1,1)$. It follows that $v_{G}=(0,0,1,0)$ or $%
v_{G}=(0,0,0,1).$ If $v_{G}=(0,0,1,0)$, then $b_{n}\geq a_{n-1}$ for all $%
n\geq 2$. Define the a SNRE as follows. 
\begin{equation*}
\left\{ 
\begin{array}{l}
d_{n}=d_{n-1}e_{n-1}, \\ 
e_{n}=d_{n-1}e_{n-1}, \\ 
d_{2}=2\text{, }e_{2}=1.%
\end{array}%
\right.
\end{equation*}%
It is evident that $a_{n}\geq d_{n}$ and $b_{n}=e_{n}$ for $n\geq 2$. Thus
we have $d_{n}=d_{n-1}^{2}$ with $d_{2}=2$. A straightforward examination
asserts that $h(\mathcal{B})=\ln 2$. Finally, $v_{F}=(0,1,1,0)$ infers that $%
v_{G}=(0,1,0,0)$ or $v_{G}=(0,0,1,0)$. Both cases can be analyzed
analogously as above.
\end{remark}

Corollary \ref{Cor : 1} is generalized from Lemma \ref{Lma : 2} based on the
discussion in Remark \ref{Rmk: 1}.

\begin{corollary}
\label{Cor : 1} If the SNRE corresponding to a given $\mathcal{B}$ is of the
dominant-type with $n_{F}\geq 2$, then $h(\mathcal{B})=\ln 2$.
\end{corollary}

\begin{example}
Suppose the basic set $\mathcal{B}$ is given by 
\begin{equation*}
\mathcal{B}=\{(1,1,1),(1,2,2),(2,2,2)\}.
\end{equation*}%
It follows that the indicator vectors of the corresponding SNRE are $%
v_{F}=(1,0,0,1)$ and $v_{G}=(0,0,0,1)$, and the SNRE is of the
dominant-type. Corollary \ref{Cor : 1} elaborates that $h(\mathcal{B})=\ln 2$%
.
\end{example}

The following corollary comes immediately from Corollary \ref{Cor : 1}, with
the proof omitted.

\begin{corollary}
\label{Cor : 2} If the SNRE corresponding to a given $\mathcal{B}$ is of the
form $v_{F}=(1,\ast ,\ast ,\ast )$ with $n_{F}\geq 2$ or $v_{G}=(\ast ,\ast
,\ast ,1)$ with $n_{G}\geq 2$. Then $h(\mathcal{B})=\ln 2$.
\end{corollary}

Corollary \ref{Cor : 2} demonstrates that positive entropy derives from $F$
(resp. $G$) containing at least two terms and one of which is $a_{n-1}^{2}$
(resp. $b_{n-1}^{2}$). To illustrate the entropy of tree-shifts of finite
type completely, we introduce another type of SNRE. Two vectors $v$ and $w$
are \emph{complementary} if $v+w$ dominates $e_{4}$, where $e_{4}=(1,1,1,1)$%
. An SNRE obtained from a given basic set $\mathcal{B}$ is of the \emph{%
complementary-type} if the corresponding indicator vectors $v_{F}$ and $%
v_{G} $ are complementary.

\begin{lemma}
\label{Lma : 1} Suppose a basic set $\mathcal{B}$ is given. If the
corresponding SNRE is of complementary-type, then $h(\mathcal{B})=\ln 2$.
\end{lemma}

\begin{proof}
If the SNRE is of the complementary-type. Then we have 
\begin{equation*}
c_{n}=a_{n}+b_{n}\geq a_{n-1}^{2}+2a_{n-1}b_{n-1}+b_{n-1}^{2}=c_{n-1}^{2}
\end{equation*}%
with $c_{2}=a_{2}+b_{2}\geq 4$. Suppose the sequence $\{d_{n}\}$ satisfies $%
d_{n}=\left( d_{n-1}\right) ^{2}$ and $d_{2}=4.$ It can be easily checked
that $c_{n}\geq d_{n}$ for all $n\geq 2$. Hence, we have 
\begin{equation*}
h(\mathcal{B})=\lim_{n\rightarrow \infty }\frac{\ln ^{2}c_{n}}{n}\geq
\lim_{n\rightarrow \infty }\frac{\ln ^{2}d_{n}}{n}\geq \ln 2.
\end{equation*}%
This completes the proof.
\end{proof}

\begin{corollary}
\label{Cor : 3} If the SNRE corresponding to a given $\mathcal{B}$ satisfies 
$v_{F}+v_{G}\geq (1,0,1,1)$ or $(1,1,0,1)$, then $h(\mathcal{B})=\ln 2$.
\end{corollary}

\begin{proof}
Notably, 
\begin{equation*}
c_{n}=a_{n}+b_{n}\geq a_{n-1}^{2}+a_{n-1}b_{n-1}+b_{n-1}^{2}\geq \frac{1}{2}%
c_{n-1}^{2}.
\end{equation*}%
Suppose $\{d_{n}\}$ satisfies $d_{n}=\frac{1}{2}d_{n-1}^{2}$ and $d_{2}\geq
3 $. Then we have $c_{n}\geq d_{n}$, and hence $h(\mathcal{B})=\ln 2$. This
completes the proof.
\end{proof}

\begin{example}
Suppose the basic set $\mathcal{B}$ is given as 
\begin{equation*}
\mathcal{B}=\{(1,1,1),(1,1,2),(2,2,1),(2,2,2)\}.
\end{equation*}%
It is seen that the indicator vectors of the corresponding SNRE are $%
v_{F}=(1,1,0,0)$ and $v_{G}=(0,0,1,1)$. Hence the SNRE is of
complementary-type. Lemma \ref{Lma : 1} asserts that the entropy of
tree-shift $X^{\mathcal{B}}$ is $h(\mathcal{B})=\ln 2$.
\end{example}

With the discussion above, we are in the position of demonstrating the main
result of this section.

\begin{theorem}
\label{Thm: 3}Suppose $d=k=2$. Then the entropy of the tree-shift of finite
type generated from a basic set $\mathcal{B}$ is either $h(\mathcal{B})=0$
or $h(\mathcal{B})=\ln 2$.
\end{theorem}

\begin{proof}
The demonstration is implemented in cases.

\noindent \textbf{I.} $v_{F}=(1,\ast ,\ast ,\ast )$ with $n_{F}\geq 2$. This
case comes directly from Corollary \ref{Cor : 2}.

\noindent \textbf{II.} $v_{F}=(0,\ast ,\ast ,\ast )$ with $n_{F}=3$, i.e., $%
v_{F}=(0,1,1,1)$. In this case, the investigation of $v_{G}=(1,0,0,0)$ and $%
v_{G}=(0,\ast ,\ast ,\ast )$ is conducted by Lemma \ref{Lma : 1} and
Corollary \ref{Cor : 1}, respectively.

\noindent \textbf{III.} $v_{F}=(0,\ast ,\ast ,\ast )$ with $n_{F}=2$. The
Pigeonhole Principle indicates that at least one of the center two entries
of $v_{F}$ is $1$. The elucidation of $v_F = (0, 1, 0, 1)$ can be treated
similarly as the discussion of $v_f = (0, 0, 1, 1)$. The detailed
implementation is divided into several sub-cases and is listed in the tables
below.

\noindent \textbf{III-1.} $n_{G} \geq 3$.

\begin{tabular}{c|cc}
& $v_{F}=(0,1,1,0)$ & $v_{F}=(0,0,1,1)$ \\[0.2ex] \hline
&  &  \\[-2ex] 
$v_{G}=(1,1,1,1)$ & Lemma \ref{Lma : 2} & \text{Lemma \ref{Lma : 2}} \\%
[0.2ex] 
$v_{G}=(1,0,1,1)$ & Corollary \ref{Cor : 2} & \text{Corollary \ref{Cor : 2}}
\\[0.2ex] 
$v_{G}=(1,1,0,1)$ & Corollary \ref{Cor : 2} & \text{Corollary \ref{Cor : 2}}
\\[0.2ex] 
$v_{G}=(0,1,1,1)$ & Corollary \ref{Cor : 2} & \text{Corollary \ref{Cor : 2}}
\\[0.2ex] 
$v_{G}=(1,1,1,0)$ & Lemma \ref{Lma : 2} & \text{Lemma \ref{Lma : 1}}%
\end{tabular}

\noindent \textbf{III-2.} $n_{G}=2$

\begin{tabular}{c|cc}
& $v_{F}=(0,1,1,0)$ & $v_{F}=(0,0,1,1)$ \\[0.2ex] \hline
&  &  \\[-2ex] 
$v_{G}=(1,0,0,1)$ & \text{Lemma \ref{Lma : 1}} & \text{Corollary \ref{Cor :
3}} \\[0.2ex] 
$v_{G}=(1,0,1,0)$ & (1) & \text{Corollary \ref{Cor : 3}} \\[0.2ex] 
$v_{G}=(1,1,0,0)$ & (2) & \text{Lemma \ref{Lma : 1}} \\[0.2ex] 
$v_{G}=(0,1,0,1)$ & \text{Corollary \ref{Cor : 2}} & \text{Corollary \ref%
{Cor : 2}} \\[0.2ex] 
$v_{G}=(0,1,1,0)$ & $a_{n}=b_{n}$ for all $n\geq 2$ & (3) \\[0.2ex] 
$v_{G}=(0,0,1,1)$ & \text{Corollary \ref{Cor : 2}} & \text{Corollary \ref%
{Cor : 2}}%
\end{tabular}

\noindent \textbf{III-3.} $n_{G}=1$

\begin{tabular}{c|cc}
& $v_{F}=(0,1,1,0)$ & $v_{F}=(0,0,1,1)$ \\[0.2ex] \hline
&  &  \\[-2ex] 
$v_{G}=(1,0,0,0)$ & (4) & \text{Corollary \ref{Cor : 3}} \\[0.2ex] 
$v_{G}=(0,1,0,0)$ & \text{Lemma \ref{Lma : 2}} & (5) \\[0.2ex] 
$v_{G}=(0,0,1,0)$ & \text{Lemma \ref{Lma : 2}} & \text{Lemma \ref{Lma : 2}}
\\[0.2ex] 
$v_{G}=(0,0,0,1)$ & (6) & \text{Lemma \ref{Lma : 2}}%
\end{tabular}

It only remains to deal with Cases (1)-(6). Notably, Cases (1)-(3) can be
treated similarly. The equation for Case (1) is 
\begin{equation*}
\left\{ 
\begin{array}{l}
a_{n}=2a_{n-1}b_{n-1}, \\ 
b_{n}=a_{n-1}^{2}+a_{n-1}b_{n-1}, \\ 
a_{2}=b_{2}=2.%
\end{array}%
\right.
\end{equation*}%
We claim that $a_{n}=b_{n}$ for all $n\geq 2$. Observe that $a_{2}=b_{2}$.
Suppose $a_{k}=b_{k}$ for some $2\leq k\in \mathbb{N}$. Then 
\begin{equation*}
a_{k+1}=2a_{k}b_{k}=a_{k}b_{k}+a_{k}b_{k}=a_{k}^{2}+a_{k}b_{k}=b_{k+1}.
\end{equation*}%
The claim follows by mathematical induction. Thus the system is reduced to
solving $a_{n}$ which satisfies $a_{n}=2a_{n-1}^{2}$ and $a_{2}=2$. Hence we
have $h(\mathcal{B})=\ln 2$.

In Case (4), we have the corresponding SNRE as follows. 
\begin{equation*}
\left\{ 
\begin{array}{l}
a_{n}=2a_{n-1}b_{n-1}, \\ 
b_{n}=a_{n-1}^{2}, \\ 
a_{2}=2,\text{ }b_{2}=1.%
\end{array}%
\right.
\end{equation*}%
This is equivalent to solving $a_{n}=2a_{n-1}a_{n-2}^{2}$ with $a_{2}=2.$
Suppose $\{d_{n}\}$ is defined as 
\begin{equation*}
d_{n}=d_{n-1}d_{n-2}^{2},\quad d_{2}=2.
\end{equation*}%
It is seen that $a_{n}\geq d_{n}$ for all $n\geq 2$. Let $d_{n}=\alpha
^{\beta ^{n}}$ we then have $\beta ^{2}-\beta -2=0$, which implies $\beta =2$%
, i.e., $h(\mathcal{B})=\ln 2$. This completes the discussion of Case (4).

In Case (5), we have the following SNRE 
\begin{equation*}
\left\{ 
\begin{array}{l}
a_{n}=a_{n-1}b_{n-1}+b_{n-1}^{2}, \\ 
b_{n}=a_{n-1}b_{n-1}, \\ 
a_{2}=2\text{, }b_{2}=1.%
\end{array}%
\right.
\end{equation*}%
Since $a_{n}\geq b_{n}$ for all $n\geq 2$. Using the same method of the
proof in Lemma \ref{Lma : 2} and Remark \ref{Rmk: 1} we have $h(\mathcal{B}%
)=\ln 2$. Finally, the SNRE of Case (6) is as follows.

\begin{equation*}
\left\{ 
\begin{array}{l}
a_{n}=2a_{n-1}b_{n-1}, \\ 
b_{n}=b_{n-1}^{2}, \\ 
a_{2}=2\text{, }b_{2}=1.%
\end{array}%
\right.
\end{equation*}%
It is evident that $b_{n}=1$ for $n\geq 2$. Thus we have $a_{n}=2a_{n-1}$
for $a_{2}=2$. It is a simple matter to see that $h(\mathcal{B})=0$. The
proof is thus complete.
\end{proof}


\section{Entropy for $d, k\geq 2$}

This section investigates the entropy of tree-shifts of finite type defined
on infinite trees, in which each node has exactly $d$ children for $d\geq 2$%
, with labels in $\mathcal{A}=\{1,\ldots ,k\}$ of cardinality $k\geq 2$. The
corresponding SNRE becomes 
\begin{equation}
\left\{ 
\begin{array}{l}
a_{n}^{(1)}=F^{(1)}(a_{n-1}^{(1)},\ldots ,a_{n-1}^{(k)}), \\ 
\vdots \\ 
a_{n}^{(k)}=F^{(k)}(a_{n-1}^{(1)},\ldots ,a_{n-1}^{(k)}), \\ 
a_{2}^{(i)}=a_{i}=\left\vert B_{2}(X_{i}^{\mathcal{B}})\right\vert \text{
for }i=1,\ldots ,k,%
\end{array}%
\right.  \label{3}
\end{equation}%
where $F^{(i)}$ is a binary combination over $\{a_{n-1}^{(1)},\ldots ,a_{n-1}^{(k)}\}^d$ for $1\leq i\leq k$.


\subsection{Realization theorem}

The last section indicates that, for the case where $d=k=2$, the entropy of
a tree-shift of finite type is either $0$ or $\ln 2$, provided the forbidden
set consists of $2$-blocks. Recall that the class of such tree-shifts of
finite type is also known as Markov tree-shifts \cite{AB-TCS2012}. Roughly
speaking, a binary Markov tree-shift labeled by two symbols is of either the
simplest or the most complicated dynamical behavior. There is no Markov
tree-shift in between. It is then natural to ask if this is the case in
general.

In this section, we demonstrate that tree-shifts are capable of rich
dynamical behavior by showing the entropy set of tree-shifts of finite type
covers a large number of reals which include the multinacci numbers.

\begin{theorem}[Realization Theorem]
\label{Thm : 2} Let $\rho $ be the maximal root of $x^{n}-%
\sum_{i=1}^{n-1}k_{i}x^{i}=0$, where $k_{i}\in \mathbb{N}\cup \{0\}$ for $%
i=1,\ldots ,n-1$. Then there exist $d,k\geq 2$ and $\mathcal{B}$ such that $%
h(\mathcal{B})=\ln \rho $.
\end{theorem}

\begin{proof}
Assume that $\rho $ is the maximal root of 
\begin{equation}
F(x)=x^{p}-k_{1}x^{p_{1}}-k_{2}x^{p_{2}}-\cdots
-k_{l-1}x^{p_{l-1}}-k_{l}x^{p_{l}},\text{ }k_{i}\in \mathbb{N},\text{ }%
\forall 1\leq i\leq l,  \label{4}
\end{equation}%
where $p_{l}=0$. Let $q_{i}=p-p_{i}$ for $i=1,\ldots ,l$. Thus $\rho $
satisfies the equation 
\begin{equation}
1=\sum_{j=1}^{l}k_{j}x^{-q_{j}}.  \label{14}
\end{equation}%
The proof is divided into two parts.

\textbf{1.} $p_{1}<p-1$. Introduce the symbol set $\bigcup_{j=1}^{l}%
\bigcup_{i=0}^{q_{j}-1}\{a^{(j,i)}\}\bigcup \{a^{(0)},b\}$ and denote $%
a^{(j)}=a^{(j,0)}$ for all $1\leq j\leq l$. Define $d=\sum_{i=1}^{l}k_{i}$
and $k=2+\sum_{j=1}^{l}q_{j}$, we construct the SNRE according to $d$ and $k$
as follows. 
\begin{equation}
\left\{ 
\begin{array}{l}
a_{n}^{\left( 0\right) }=2\left( a_{n-1}^{\left( 1\right) }\right)
^{k_{1}}\left( a_{n-1}^{\left( 2\right) }\right) ^{k_{2}}\left(
a_{n-1}^{\left( 3\right) }\right) ^{k_{3}}\cdots \left( a_{n-1}^{\left(
l\right) }\right) ^{k_{l}}, \\ 
a_{n}^{\left( j,r\right) }=a_{n-1}^{\left( j,r+1\right) }b_{n-1}^{d-1}\text{%
, for all }1\leq j\leq l\text{, }0\leq r\leq q_{j}-2, \\ 
a_{n}^{\left( j,q_{j}-1\right) }=a_{n-1}^{(0)}b_{n-1}^{d-1}, \\ 
b_{n}=b_{n-1}^{d}, \\ 
a_{2}^{(0)}=2,\text{ }b_{2}=a_{2}^{\left( j,r\right) }=1,\text{ }\forall
1\leq j\leq l\text{ and }0\leq r\leq q_{j}-1.%
\end{array}%
\right.  \label{6}
\end{equation}%
We claim that $h(\mathcal{B})=\ln \rho $. Since for $1\leq j\leq l$ and $%
1\leq r\leq q_{j}-2$, 
\begin{eqnarray}
a_{n}^{(j)}
&=&a_{n}^{(j,0)}=a_{n-1}^{(j,1)}b_{n-1}^{d-1}=a_{n-2}^{(j,2)}b_{n-2}^{d-1}b_{n-1}^{d-1}
\notag \\
&=&\cdots =a_{n-r}^{(j,r)}\left( b_{n-r}\cdots b_{n-1}\right) ^{d-1}  \notag
\\
&=&a_{n-(q_{j}-1)}^{(j,q_{j}-1)}\left( b_{n-(q_{j}-1)}\cdots b_{n-1}\right)
^{d-1}\text{.}  \label{10}
\end{eqnarray}%
Combining (\ref{10}) with the fact that $a_{n}^{\left( j,q_{j}-1\right)
}=a_{n-1}^{(0)}b_{n-1}^{d-1}$ we obtain 
\begin{equation}
a_{n}^{(j)}=a_{n-q_{j}}^{(0)}\left( b_{n-q_{j}}\cdots b_{n-1}\right) ^{d-1}%
\text{ for all }1\leq j\leq l\text{.}  \label{11}
\end{equation}%
Substituting (\ref{11}) in the first equation of (\ref{6}) yields 
\begin{equation}
a_{n}^{(0)}=2\prod_{j=1}^{l}\left( a_{n-q_{j}}^{(0)}\right) ^{k_{j}}\left(
b_{n-q_{j}}\cdots b_{n-1}\right) ^{k_{j}(d-1)}\text{.}  \label{12}
\end{equation}%
One can easily check that $b_{n}=1$ for all $n\geq 1$, thus we can rewrite (%
\ref{12}) as follows. 
\begin{equation*}
a_{n}^{(0)}=2\prod_{j=1}^{l}\left( a_{n-q_{j}}^{(0)}\right) ^{k_{j}}\text{
with }a_{2}^{(0)}=2\text{.}
\end{equation*}%
Taking $x_{n}=\ln a_{n}^{(0)}$ becomes 
\begin{equation*}
x_{n}=\ln 2+\sum\limits_{j=1}^{l}k_{j}x_{n-q_{j}}.
\end{equation*}%
Let $x_{n}=\alpha ^{n}$. We have the following equation 
\begin{equation}
\alpha ^{n}=\ln 2+\sum_{j=1}^{l}k_{j}\alpha ^{n-q_{j}}.  \label{8}
\end{equation}%
Since $\alpha >1$, dividing \eqref{8} by $\alpha ^{n}$ and letting $n$
approach to infinity infers 
\begin{equation*}
1=\sum_{j=1}^{l}k_{j}\alpha ^{-q_{j}}\text{.}
\end{equation*}%
That is, $\alpha =h(\mathcal{B})$ satisfies (\ref{14}).

\textbf{2.} $p_{1}=p-1$. The same proof works when we modify the first
equation of (\ref{6}) as follows.%
\begin{equation*}
a_{n}^{\left( 0\right) }=2\left( a_{n-1}^{\left( 0\right) }\right)
^{k_{1}}\left( a_{n-1}^{\left( 1\right) }\right) ^{k_{2}}\left(
a_{n-1}^{\left( 2\right) }\right) ^{k_{3}}\cdots \left( a_{n-1}^{\left(
l\right) }\right) ^{k_{l}}\text{.}
\end{equation*}%
This completes the proof.
\end{proof}

It is worth pointing out that the \emph{multinacci number} of order $n\in 
\mathbb{N}\backslash \{1\}$ is the real number $\gamma _{n}\in (1,2)$ which
is the unique positive real solution of the equation $1=x^{-1}+x^{-2}+\cdots
+x^{-n}$. It follows immediately from Theorem \ref{Thm : 2} that the
multinacci numbers of all order can be realizaed by tree-shifts of finite
type. The smallest multinacci number is the golden ratio $(1+\sqrt{5})/2$.
The following example presents a tree-shift of finite type which realizes
such number. \emph{\ }

\begin{example}
Let $g=(1+\sqrt{5})/2$ be the maximal root of the polynomial 
\begin{equation}
x^{2}-x-1=0.  \label{7}
\end{equation}%
In this case we have $p_{1}=p-1$ (defined in the proof of Theorem \ref{Thm :
2}). Construct the SNRE as follows. 
\begin{equation*}
\left\{ 
\begin{array}{l}
a_{n}^{(0)}=2a_{n-1}^{(0)}a_{n-1}^{(1)}, \\ 
a_{n}^{(1)}=a_{n-1}^{(0)}b_{n-1}, \\ 
b_{n}=b_{n-1}^{2}, \\ 
a_{2}^{(1)}=2,\text{ }b_{2}=a_{2}^{(2)}=1.%
\end{array}%
\right.
\end{equation*}%
Theorem \ref{Thm : 2} shows that $h(\mathcal{B})=\ln g$. In this case, the
desired pair of $d$ and $k$ are $d=2$ and $k=3$.
\end{example}


\subsection{Computation of entropy for general $d$ and $k$}

For SNRE (\ref{3}) and $1\leq i\leq k$, we associate an indicator vector $%
v^{(i)}:=v_{F^{(i)}}$ as defined earlier. If there exists an $l$ such that $%
v^{(l)}\geq v^{(j)}$ for all $1\leq j\leq k$, then we call $v^{(l)}$ a \emph{%
dominant vector}. We say that SNRE (\ref{3}) is of the \emph{dominant-type}
if there exists a dominant vector $v\in \{v^{(i)}\}_{i=1}^{k}$. 

\begin{example}
\label{Ex: 1}Let $d=2$ and $k=3$. Given $\mathcal{B}$ and suppose the
corresponding SNRE is of the form. 
\begin{equation}
\left\{ 
\begin{array}{l}
a_{n}^{(1)}=\left( a_{n-1}^{(1)}\right)
^{2}+a_{n-1}^{(1)}a_{n-1}^{(2)}+a_{n-1}^{(1)}a_{n-1}^{(3)}=F^{(1)}, \\ 
a_{n}^{(2)}=a_{n-1}^{(1)}a_{n-1}^{(2)}+a_{n-1}^{(1)}a_{n-1}^{(3)}=F^{(2)} \\ 
a_{n}^{(3)}=a_{n-1}^{(1)}a_{n-1}^{(3)}=F^{(3)} \\ 
a_{2}^{(1)}=3,\text{ }a_{2}^{(2)}=2\text{ and }a_{2}^{(3)}=1\text{. }%
\end{array}%
\right.  \label{9}
\end{equation}%
Then one can easily check that $v^{(1)}=(1,1,1,0,0,0,0,0,0)$ is a dominant
vector and the SNRE (\ref{9}) is of the dominant-type.
\end{example}

\begin{proposition}
\label{prop: lower-bound-entropy} Given $\mathcal{B}$, if the corresponding
SNRE (\ref{3}) is of the dominant-type with $v^{(l)}$ being a dominant
vector. Then $h(\mathcal{B})=\ln h$ or $0$, where $h$ is the maximal degree
of $a_{n-1}^{(l)}$ in $F^{(l)}$. Furthermore, if $a_{n-1}^{(l)}$ has degree $%
d$ in $F^{(l)}$, then $h(\mathcal{B})=\ln d$.
\end{proposition}

\begin{proof}
There is no loss of generality in assuming $l=1$ and suppose the maximal
degree of $a_{n-1}^{(1)}$ in $F^{(1)}$ is $h$. Following the same method as
the proof of Lemma \ref{Lma : 2} we have $h(\mathcal{B})\geq \ln h$. Thus it
suffices to show that $h(\mathcal{B})\leq \ln h$. Since $v^{(1)}$ is
dominant, we have $a_{n}^{(1)}\geq a_{n}^{(i)}$ for $i=2,\ldots ,k$ and $%
n\geq 2$. Thus $\sum_{i=1}^{k}a_{n}^{(i)}\leq ka_{n}^{(1)}$ and 
\begin{equation*}
h(\mathcal{B})=\lim_{n\rightarrow \infty }\frac{\ln
^{2}\sum_{i=1}^{k}a_{n}^{(i)}}{n}\leq \lim_{n\rightarrow \infty }\frac{\ln
^{2}a_{n}^{(1)}}{n}=\ln h\text{.}
\end{equation*}%
The last equality follows from the same argument to the proof of Lemma %
\ref{Prop : 1}, and the proof is complete.
\end{proof}

It is remarkable that Corollary \ref{Cor : 2} is a special case of
Proposition \ref{prop: lower-bound-entropy} and it is a simple matter to see
that $h(\mathcal{B})=\ln 2$ in Example \ref{Ex: 1}.

%
%

In the following, we introduce another methodology for computing entropy,
which shows that the symmetry of basic sets implies full entropy. Suppose $%
\mathcal{B}$ is a basic set consisting of two-blocks. In this case, each $u\in 
\mathcal{B}$ can be written as $u=(u_{\epsilon },u_{0},u_{1},\ldots
,u_{d-1}) $. Define a \emph{projection map} $\pi $ as 
\begin{equation*}
\pi (u)=(u_{0},u_{1},\ldots ,u_{d-1})\quad \text{for all}\quad u\in \mathcal{%
B}.
\end{equation*}%
For instance, if $u=(1,2,2,1,2,1)$ is a two-block, then $\pi (u)=(2,2,1,2,1)$%
.

\begin{definition}[Symmetric basic sets]
Given a basic set $\mathcal{B}$, we decompose it into $\mathcal{B}%
=\bigcup_{i=1}^{k}\mathcal{B}^{(i)}$, where 
\begin{equation*}
\mathcal{B}^{(i)}=\{\omega :\omega _{\epsilon }\text{ is labeled by }i\}%
\text{ for }1\leq i\leq k\text{.}
\end{equation*}%
Here we say the corresponding tree-shift $X^{\mathcal{B}}$ is \emph{symmetric%
} if $\pi (\mathcal{B}^{(i)})=\pi (\mathcal{B}^{(j)})$ for all $i\neq j$. In
this case we call the corresponding SNRE (\ref{3}) is of the \emph{%
symmetric-type}.
\end{definition}

Suppose $v^{(i)}$ is an indicator vector, we denote by $n^{(i)}=\left\vert
\left\{ j:\left( v^{(i)}\right) _{j}\neq 0\right\} \right\vert $ the number
of the non-zero coordinates of $v^{(i)}$. The following result shows that
the entropy of a symmetric tree-shift is $\ln d$.

\begin{proposition}
\label{prop:entropy-symmetric-type} If an SNRE is of symmetric-type with $%
n^{(1)}\geq 2$. Then $h(\mathcal{B})=\ln d$.
\end{proposition}

\begin{proof}
Symmetric-type property implies $a_{n}^{(i)}=a_{n}^{(j)}$ and $%
n^{(i)}=n^{(j)}$ for all $1\leq i,j\leq k$ and $n\geq 2$. Then the
corresponding SNRE (\ref{3}) can be reduced to a single equation $%
a_{n}^{(1)}=n^{(1)}\left( a_{n-1}^{(1)}\right) ^{d}$ with $n^{(1)}\geq 2$.
Hence we have $h(\mathcal{B})=\ln d$, which proves the theorem.
\end{proof}

\begin{example}
Suppose a Markov tree-shift $X^{\mathcal{B}}$ is generated by the basic set $%
\mathcal{B}=\{(i,1,1,2),(i,1,2,1),(i,2,1,1),(i,2,1,2),(i,2,2,2):i=1,2\}$.
Then the indicator vectors of the corresponding SNRE are 
\begin{equation*}
v^{(1)}=(0,1,1,0,1,1,0,1)=v^{(2)}.
\end{equation*}%
Hence the SNRE is of the symmetric-type and $n^{(1)}=5=n^{(2)}$. Proposition %
\ref{prop:entropy-symmetric-type} indicates that the entropy of $X^{\mathcal{%
B}}$ is $h(\mathcal{B})=\ln 3$.
\end{example}


\section{Entropy with Boundary Conditions}

Suppose $X$ is a tree-shift. It is natural to ask whether the topological
entropy of $X$ is influenced by the constraint of boundary conditions, and
under what conditions the topological entropy of $X$ with boundary
conditions is identical to the topological entropy of the original
tree-shift $X$.

In this section, we consider three different types of boundary conditions: 
\emph{Periodic}, \emph{Dirichlet}, and \emph{Neumann boundary conditions}.
For a one-sided shift space $X$ with finite lattice $n \in \mathbb{N}$, we
mean that $x = x_1 x_2 \ldots x_n$ for each $x \in X$. Periodic boundary
condition infers that the underlying lattice of $X$ is treated as a circle;
that is, $x_n = x_1$ for all $x \in X$. Dirichlet boundary condition
indicates that the terminal state is constant for each $x \in X$. More
precisely, $x_n = \kappa$ for all $x \in X$, where $\kappa$ is given.
Neumann boundary condition means there is zero flux in the system; namely, $%
x_n = x_{n-1}$ for all $x \in X$.

For the case where $X$ is a tree-shift and $n \geq 2$, the collection of $n$%
-blocks in $X$ with the constraint of periodic, Dirichlet, and Neumann
boundary conditions, denoted by $B_n^P(X), B_n^{D_i}(X)$, and $B_n^{N}(X)$,
respectively, is defined as 
\begin{align}
B_n^P &= \{u \in B_n(X): u_w = u_{\epsilon} \text{ for } |w| = n-1\}, \\
B_n^{D_i} &= \{u \in B_n(X): u_w = i \text{ for } |w| = n-1\}, \quad 1 \leq
i \leq k, \\
B_n^N &= \{u \in B_n(X): u_w = u_{\hat{w}} \text{ for } |w| = n-1, \hat{w} =
w_1 \cdots w_{n-2}\},
\end{align}
respectively. Based on the definitions of $B_n^P, B_n^{D_i}$, and $B_n^N$,
the topological entropy of tree-shift $X$ with boundary condition is defined
as 
\begin{equation*}
h^{\iota} (\mathcal{B}) = h^{\iota} (X^{\mathcal{B}}) = \lim_{n \to \infty} 
\frac{\ln^2 |B_n^{\iota}(X^{\mathcal{B}})|}{n}
\end{equation*}
provided the limit exists, herein $\iota = P, D_1, \ldots, D_k, N$.

In the rest of this section, $X = X^{\mathcal{B}}$ is a binary Markov
tree-shift over $\mathcal{A} = \{1, 2\}$ unless otherwise stated. Namely, we
focus on the case where $d = k = 2$.

Obviously, $h^{\iota} (\mathcal{B}) \leq h(\mathcal{B})$ for $\iota = P,
D_1, D_2, N$. Hence, the effect of boundary conditions is seen only in those
tree-shifts carrying positive topological entropy.

\begin{theorem}
\label{thm:Entropy-NBC} Suppose $X^{\mathcal{B}}$ is a Markov tree-shift and 
$h(\mathcal{B}) > 0$. Then $h^N(\mathcal{B}) = h(\mathcal{B})$ if and only
if either (1) $\{(1, 1, 1), (2, 2, 2)\} \subseteq \mathcal{B}$ or (2) $%
\mathcal{B}$ contains $\{(1, i, i), (2, i, i)\}$ for some $i = 1, 2$.
\end{theorem}

\begin{proof}
We start the proof with the ``Only If'' part. It is easily seen that $\{(1,
1, 1), (2, 2, 2)\} \subseteq \mathcal{F}$ asserts that $h^N(\mathcal{B}) = 0
< h(\mathcal{B})$. Suppose, without loss of generality, $(1, 1, 1) \in 
\mathcal{B}$ and $(2, 2, 2) \in \mathcal{F}$. If $(2, 1, 1) \in \mathcal{F}$%
, then 
\begin{equation*}
|B_n^N(X^{\mathcal{B}})| = |B_{n-1}(X^{\mathcal{B}}_1)| = 1 \quad \text{for}
\quad n \geq 2.
\end{equation*}
This contradicts to $h^N(\mathcal{B}) = h(\mathcal{B})$. Hence, $\{(1, 1,
1), (2, 1, 1)\}$ is a subset of $\mathcal{B}$.

For the ``If'' part, observe that $|B_n^N(X^{\mathcal{B}})| = |B_{n-1}(X^{%
\mathcal{B}})|$ follows from the presumption $\{(1, 1, 1), (2, 2, 2)\}
\subseteq \mathcal{B}$, and thus $h^N(\mathcal{B}) = h(\mathcal{B})$.
Without loss of generality, assume that $\{(1, 1, 1), (2, 1, 1)\} \subseteq 
\mathcal{B}$. It comes immediately that 
\begin{equation*}
|B_n^N(X^{\mathcal{B}})| \geq |B_{n-1}(X^{\mathcal{B}}_1)| = |B_{n-2}(X^{%
\mathcal{B}})|,
\end{equation*}
which leads to the desired conclusion.

This completes the proof.
\end{proof}

\begin{theorem}
\label{thm:Entropy-DBC} Suppose $X^{\mathcal{B}}$ is a Markov tree-shift and 
$h(\mathcal{B}) > 0$. Then, for $i = 1, 2$, $h^{D_i}(\mathcal{B}) = h(%
\mathcal{B})$ if and only if either (1) $\{(1, i, i), (2, i, i)\} \subseteq 
\mathcal{B}$ or (2) $\{(\bar{i}, i, i), (1, \bar{i}, \bar{i}), (2, \bar{i}, 
\bar{i})\} \subseteq \mathcal{B}$, where $i + \bar{i} = 3$.
\end{theorem}

\begin{proof}
We address the elucidation for the case where $i = 1$. The other case can be
done analogously.

Start with the ``If'' part. Note that 
\begin{equation*}
B_n^{D_1}(X^{\mathcal{B}}) = \{u \in B_n(X^{\mathcal{B}}): u_w = 1 \text{
for } |w| = n-1\}.
\end{equation*}
Suppose $\{(1, 1, 1), (2, 1, 1)\} \subseteq \mathcal{B}$. It follows
immediately that 
\begin{equation*}
|B_n^{D_1}(X^{\mathcal{B}})| = |B_{n-1}(X^{\mathcal{B}})|.
\end{equation*}
Hence, 
\begin{equation*}
h^{D_1}(\mathcal{B}) = \lim_{n \to \infty} \frac{\ln^2 |B_n^{D_1}(X^{%
\mathcal{B}})|}{n} = \lim_{n \to \infty} \frac{\ln^2 |B_{n-1}(X^{\mathcal{B}%
})|}{n} = h(\mathcal{B}).
\end{equation*}
On the other hand, $\{(2, 1, 1), (1, 2, 2), (2, 2, 2)\} \subseteq \mathcal{B}
$ infers that 
\begin{equation*}
|B_n^{D_1}(X^{\mathcal{B}})| = |B_{n-1}^{D_2}(X^{\mathcal{B}})| =
|B_{n-2}(X^{\mathcal{B}})|,
\end{equation*}
which concludes the coincidence of $h^{D_1}(\mathcal{B})$ and $h(\mathcal{B}%
) $.

Conversely, $\{(1, 1, 1), (2, 1, 1)\} \subseteq \mathcal{F}$ derives that $%
B_n^{D_1}(X^{\mathcal{B}}) = \varnothing$ for all $n \geq 2$. Hence $0 =
h^{D_1}(\mathcal{B}) < h(\mathcal{B})$. Without loss of generality, we
assume that $(1, 1, 1) \in \mathcal{B}$ while $(2, 1, 1) \in \mathcal{F}$.
Observe that 
\begin{equation*}
|B_n^{D_1}(X^{\mathcal{B}})| = |B_{n-1}^{D_2}(X^{\mathcal{B}})| = |B_2(X^{%
\mathcal{B}}_1)|
\end{equation*}
since $(2, 1, 1) \in \mathcal{F}$. This shows that $0 = h^{D_1}(\mathcal{B})
< h(\mathcal{B})$. Therefore, we have $(2, 1, 1) \in \mathcal{B}$. Analogous
argument demonstrates that $0 = h^{D_1}(\mathcal{B}) < h(\mathcal{B})$ if $%
\{(1, 2, 2), (2, 2, 2)\} \nsubseteq \mathcal{B}$.

The proof is then complete.
\end{proof}

Theorem \ref{thm:1to1-correspondence-entropy-SNRE} demonstrates that the
topological entropy of $X$ is associated with a system of nonlinear
recurrence equations 
\begin{equation*}
\left\{ 
\begin{array}{l}
a_{n}^{(1)}=F^{(1)}(a_{n-1}^{(1)},a_{n-1}^{(2)}), \\ 
a_{n}^{(2)}=F^{(2)}(a_{n-1}^{(1)},a_{n-1}^{(2)}),, \\ 
a_{2}=|B_{2}(X_{1}^{\mathcal{B}})|,b_{2}=|B_{2}(X_{2}^{\mathcal{B}})|.%
\end{array}%
\right.
\end{equation*}%
for some $F^{(1)}$ and $F^{(2)}$, where $a_{n}^{(i)}=|B_{n}(X_{i}^{\mathcal{B}})|$ for $n\geq 3$ and $i=1,2$. Let $%
v_{i}$ be the indicator vector of $F^{(i)}$ for $i=1,2$. Theorem \ref%
{thm:Entropy-PBC} indicates a sufficient condition for the coincidence of
topological entropy of $X$ with and without periodic boundary condition.

\begin{theorem}
\label{thm:Entropy-PBC} Suppose $X^{\mathcal{B}}$ is a Markov tree-shift and 
$h(\mathcal{B}) > 0$.

\begin{enumerate}
\item $h^P(\mathcal{B}) = h(\mathcal{B})$ if $v_i$ dominates $v_{\bar{i}}$
and $\{(1, i, i), (2, i, i)\} \subseteq \mathcal{B}$ for some $i = 1, 2$,
where $i + \bar{i} = 3$.

\item If $h^P(\mathcal{B}) = h(\mathcal{B})$, then $\{(1, i, i), (2, i, i)\}
\subseteq \mathcal{B}$ for some $i = 1, 2$.
\end{enumerate}
\end{theorem}

\begin{proof}
(a) Since $v_i$ dominates $v_{\bar{i}}$, we have 
\begin{equation*}
h(\mathcal{B}) = \lim_{n \to \infty} \frac{\ln^2 |B_n(X^{\mathcal{B}})|}{n}
= \lim_{n \to \infty} \frac{\ln^2 |B_n(X^{\mathcal{B}}_i)|}{n}.
\end{equation*}
Moreover, $\{(1, i, i), (2, i, i)\} \subseteq \mathcal{B}$ infers that 
\begin{equation*}
|B_n^P(X^{\mathcal{B}})| \geq |B_n^P(X^{\mathcal{B}}_i)| = |B_{n-1}(X^{%
\mathcal{B}}_i)|.
\end{equation*}
The desired result then follows.

(b) Suppose, for each $i \in \{1, 2\}$, there exists $i^{\prime }\in \{1,
2\} $ such that $(i^{\prime }, i, i) \in \mathcal{F}$. A routine
verification demonstrates that $\{(i, 1, 1), (i, 2, 2)\} \subseteq \mathcal{B%
}$ and $\{(\bar{i}, 1, 1), (\bar{i}, 2, 2)\} \subseteq \mathcal{F}$ yields $%
h^P(\mathcal{B}) = 0 < h(\mathcal{B})$ since $|B_n^P(X^{\mathcal{B}})| \leq
2 $ for all $n$, which derives contradiction. Without loss of generality, we
assume that 
\begin{equation*}
\{(1, 1, 1), (2, 2, 2)\} \subseteq \mathcal{B} \quad \text{and} \quad \{(2,
1, 1), (1, 2, 2)\} \subseteq \mathcal{F}.
\end{equation*}
We can conclude that $|B_n^P(X^{\mathcal{B}})| = 2$ for all $n$, which leads
to $h^P(\mathcal{B}) = 0 < h(\mathcal{B})$. Therefore, $h^P(\mathcal{B}) = h(%
\mathcal{B})$ implies $\{(1, i, i), (2, i, i)\} \subseteq \mathcal{B}$ for
some $i = 1, 2$.
\end{proof}


\section{Conclusion and Remarks}


\subsection{Results}

This paper studies the complexity of tree-shifts of finite type via a
well-known indicator like entropy. The main difference between the definition of
entropy for classical symbolic dynamics and tree-shifts is the power that
apply to the natural logarithmic function. The entropy of binary Markov
tree-shifts over two symbols, i.e., $d=k=2$, is characterized completely. It
turns out the entropy of binary Markov tree-shifts over two symbols is
either $0$ or $\ln 2$. In other words, the dynamical behavior of a binary
Markov tree-shift with two symbols is either the simplest or the most
complicated.

For the general case, i.e., $\max \{d,k\}\geq 3$, we show that, the entropy
set of tree-shifts of finite type covers a large number of reals which
include the multinacci numbers (Realization Theorem). More importantly, such
tree-shift can be constructed explicitly. Finally, the entropies of some
tree-shifts of finite type can be computed explicitly (dominant-type,
uniformly mixing and symmetric-type).

Whenever we restrict the case where $d = k = 2$, the necessary and
sufficient condition for the coincidence of the entropy with and without
boundary condition is elaborated. Herein, we consider three different types
of boundary conditions, say, the periodic, Dirichlet, and Neumann boundary
conditions. It is remarkable that the invariance of entropy with boundary
constraints is determined by the basic set of allowed patterns.


\subsection{Open problems}

There are several interesting open problems related to this paper. Some of
them are in preparation and will be discussed in the future work.

\begin{problem}
How to compute the hidden entropy explicitly for $d=k=2$?
\end{problem}

\begin{problem}
How to compute the entropy for $d,k\geq 2$?
\end{problem}

To the best of our knowledge, there is no algorithm for solving the entropy
and hidden entropy of an SNRE. Many researchers have devoted to the study of nonlinear recurrence equations. Readers are referred to \cite{GreathouseIV-2012} for more details.

\begin{problem}
Can the entropy of $1$-$d$ (or $2$-$d$) SFTs be realized by TSFT? That is,
if $\ln \rho $ is the topological entropy of some matrix shift $X=X_{A}$ (or 
$X=X_{A,B}$), where $A$ and $B$ are the corresponding transition matrices.
Does there exist a TSFT with $h(\mathcal{B})=\ln \rho $? One can check that
the number $\rho $ in Theorem \ref{Thm : 2} is a class of topological
entropy of $1$-$d$ SFTs. But the construction of the TSFT in the proof of
Theorem \ref{Thm : 2} can not extend to the general cases.
\end{problem}

In \cite{BC-2015}, we show that every tree-shift of finite type is
topological conjugate to a vertex tree-shift, which is defined analogously
to the definition of matrix shifts. We conjecture that the answer to the
above problem is affirmative, and the related work is in preparation.

Last but not least, Section 4 elaborates, for the case where $d = k = 2$,
the necessary and sufficient conditions for the coincidence of entropy of
Markov tree-shifts with and without boundary condition. Herein, three
different types of frequently considered boundary conditions are discussed,
that is to say, periodic, Dirichlet, and Neumann boundary conditions. Except
for the invariance problem, it is natural to ask the following: How to compute
the entropy of tree-shifts of finite type with boundary condition?

\begin{problem}
How to compute the exact value of the entropy of tree-shifts of finite type
with boundary condition?
\end{problem}

\section*{Acknowledgment}
We would like to express our deep gratitude for the anonymous referees' valuable and constructive comments, which have significantly improved the quality and readability of this paper.

\bibliographystyle{amsplain}
\bibliography{../../grece}


\end{document}